\documentclass{article}
\usepackage[utf8]{inputenc}
\usepackage{graphicx,amsrefs}
\usepackage{amsmath,amsthm,amssymb,mathtools}
\usepackage{comment,enumerate}

 \newcommand\figref{Figure~\ref}

 \usepackage{comment,amsmath,gensymb,booktabs}
\usepackage{dsfont}\let\mathbb=\mathds

\usepackage{standalone}
\usepackage{mathtools,amsmath}
\usepackage{graphicx}
\usepackage{enumerate}
\usepackage{tikz}
\usepackage{color}
\usepackage{comment}
\usepackage{bibentry}
\usepackage{pgfplots}
\usetikzlibrary{pgfplots.groupplots}

\usetikzlibrary{fit}
\pgfplotsset{compat=1.9}
\usetikzlibrary{positioning}

\providecommand{\norm}[1]{\lVert#1\rVert}
\newcommand{\red}{\textcolor{red}}

\usepackage{graphicx,amsthm}
\usepackage{footnote}
\begin{document}
\title{On Closed-form H-infinity Optimal Control \\ and Large-scale Systems}
\author{Carolina~Bergeling, Richard Pates 
    and Anders~Rantzer \footnote{C. Bergeling, R. Pates and A. Rantzer are with the Department
of Automatic Control, Lund University, Box 118, SE-221 00 Lund, Sweden. Phone: +46 46 222 87 87. e-mail: carolina.bergeling@control.lth.se, rantzer@control.lth.se, richard.pates@control.lth.se. Fax: +46 46 13 81 18. The authors are members of the LCCC Linnaeus Center and the ELLIIT Excellence Center at Lund University. Corresponding author: C. Bergeling, carolina.bergeling@control.lth.se.}}
\maketitle


\begin{abstract}
We address a class of systems for which the solution to an H-infinity optimal control problem can be given on a very simple closed form. In fact, both the control law and optimal performance value are explicitly given. The class of systems include models for large-scale systems such as temperature dynamics in buildings, buffer networks and transportations systems. Furthermore, the structure of the control law is suitable for distributed control of such large-scale systems, which is illustrated through examples. 
\end{abstract}


\newtheorem{theo}{Theorem}
\setcounter{theo}{0}
\newtheorem{lem}{Lemma}
\setcounter{lem}{0}
\newtheorem{example}{Example}
\setcounter{example}{0}
\newtheorem{remark}{Remark}
\setcounter{remark}{0}
\newtheorem{coro}{Corollary}
\setcounter{coro}{0}
\newtheorem{defin}{Definition}
\setcounter{defin}{0}
\newtheorem{ass}{Assumption}
\setcounter{ass}{0}
\newtheorem{problem}{Problem}
\setcounter{problem}{0}
\newtheorem{definition}{Definition}
\setcounter{definition}{0}
\newtheorem{prop}{Proposition}
\setcounter{prop}{0}

%

\section{Introduction}
%
%
%
%

Classical methods for controller synthesis are often inadequate for large-scale systems. That is, systems with large numbers of sensors and actuators, often with limited information sharing. This is in part due to their computational complexity. Furthermore, the controllers they render do rarely comply with the apparent interconnection-restrictions.

The fields of decentralized and distributed control have emerged as a response to this inadequacy. The work on quadratic invariance has given tools for identifying control problem with structural constraints, as those that result from interconnection restrictions, that are convex \cite{lessard2016convexity,rotkowitz2006characterization}. However, even if a structured control problem can be posed as a convex problem, it does not mean that it is easy to solve for large-scale systems. This is addressed in \cite{dhingra2016sparsity} where the invariant properties of the considered system are utilized in order to gain convexity and computational efficiency.  On a similar note, \cite{tanaka2011bounded,rantzer2015scalable} treat distributed and scalable control of positive systems, wherein the methods scale well as the system's size grows. However, these methods do not guarantee optimality. 

Spatially invariant systems are treated in \cite{bamieh2002distributed,bamieh2005convex}, where it is shown that the classical synthesis approaches can result in controllers that are also spatially invariant. That is, the structure of the plant is inherited by the controller, which could be beneficial in large-scale applications. Another approach to distributed design is to give requirements on system components, such that a global performance criterion can be met regardless of their interconnection. One can also use a certain notion of separability of the constrained large-scale optimal control problem as described in \cite{wang2017separable}, \cite{wang2016system} to show that the synthesis problem can be solved locally, and thus via parallel computation. In \cite{pates2016scalable}, a slightly different approach is taken where optimality is not of main focus. Instead, it is shown that completely independent design of control laws, by standard synthesis techniques, using only local information can be used as long as each controller meets a certain requirement. The performance degradation, when only local information is available in synthesis, is investigated in \cite{delvenne2006price}. In some cases, it is show to be at least twice the global optimum, i.e., when the global information of the system is available in synthesis. 

We address the inapplicability of $H_{\infty}$ synthesis for large-scale systems. In particular, the problem of $H_{\infty}$ optimal control. If an $H_{\infty}$ optimal control problem is solvable, it can have several optimal solutions. That is, there exists several optimal control laws. Naturally, they will have different properties. As already mentioned, the structural sparsity and computational complexity of the controller are of importance in the design of controllers for large-scale systems. Therefore, the choice of an optimal controller is crucial. 

We present a $H_{\infty}$ control law on closed-form, that is shown to be optimal for a large class of systems. This class includes many models of large-scale such as those that describe the dynamics of temperature distribution in buildings, buffer systems and transportation networks. The control law is in no need of computation as it is explicitly given. This is a rarity in $H_{\infty}$ control. Furthermore, it complies with the interconnection properties of the considered system, and can be shown to be suitable for distributed control. We will illustrate this result by means of the following example. Consider a control system comprised of $N$ subsystems with dynamics described by
\begin{equation*} 
\begin{aligned}
\dot{x}_i &= -a_ix_i + \sum_{(i,j) \in {\mathcal{E}}}u_{ij}-u_{ji}+w_i,
\end{aligned} 
\end{equation*}
where $x_i$ is the state of subsystem $i$, $w_i$ is a disturbance and the control inputs $u_{ij}$ are to be designed. Furthermore, $(i,j)$ is in the set $\mathcal{E}$ if and only if subsystems $i$ and $j$ are interconnected. Define $(w)_i = w_i$, $(x)_i = x_i$ and $(u)_k = u_{ij}$ and let $G_K(s)$ be the transfer matrix from $w$ to $x$ and $u$ when $u = Kx$. Then, the norm $\|G_K\|_{\infty}$ is minimized by $K$ that corresponds to
\begin{align*}
u_{ij} = -x_i/a_i+x_j/a_j,
\end{align*}
if $a_i > 0$ for all $i \in \{1,\dots, N\}$. In fact, we show that for a large class of systems $\dot{x} = Ax+Bu+w$, where $A$ and $B$ are matrices of appropriate dimensions, the control law $u = B^TA^{-T}x$ is optimal. The control law in the example above can be given on this form. Note that the control law comply with the interconnection policy of the system.
The general result covers a larger class of systems than that considered above. However, it still provides a control law that is \textit{distributed} and \textit{scalable} with regard to the connectivity of the system. That is, distributed in the sense that each control input only needs local information and scalable as one can add an extra control input without the need to update the existing~ones.

It is important to point out that our results are restricted to a certain class of systems and problems. This is in contrast to many of the methods reviewed above for distributed control. However, we present the system properties needed for the proposed control law to be optimal. In \cite{lidstrom2016optimal}, a preliminary version of the result of the optimal control law was given. Furthermore, treated in discrete time in~\cite{lidstromDiscrete}. However, it only considers a subclass of the systems considered here. Namely, systems with certain symmetry. In this work, we go beyond such systems. The general idea behind the synthesis procedure of our distributed control law is to consider a certain relaxation of the $H_{\infty}$ problem. In particular, a least-squares type problem in $L_2$. This as the solution might render a controller that have properties suitable for control of large-scale systems, such as sparsity and simple dynamics. Furthermore, we specify classes of systems and problem setups for which equality holds between the original problem and the relaxation.  

\subsection{Outline}
Section \ref{sec:recipe} describes the considered class of systems and problem form. Furthermore, it states the result on the proposed control law and the criteria for it to be optimal. Section \ref{sec:bottleneck} covers important cases for which the criteria are readily fulfilled. 
The scope of the results are discussed in Section~\ref{sec:scope}. Finally, in Section~\ref{sec:discussion}, the results are compared to existing methods for both general and distributed $H_{\infty}$ control. 

\subsection{Notation and preliminaries}
The set of real numbers is denoted $\mathbb{R}$, the space of {${n\textrm{-by-}m}$} real-valued matrices is denoted $\mathbb{R}^{n \times m}$ and the identity matrix is written as $I$.
The spectral norm of a matrix $M$ is denoted $\norm{M}$. A square matrix $M \in \mathbb{R}^{n \times n}$ is said to be Hurwitz if all eigenvalues have negative real part. Furthermore, for a square symmetric matrix $M$, $M \prec 0$ ($M \preceq 0$) means that $M$ is negative (semi)definite while $M \succ 0$ ($M \succeq 0$) means $M$ is positive (semi)definite. We denote the conjugate transpose of a matrix $M\in \mathbb{C}^{m\times n}$ by $M^*$. The pseudoinverse of a matrix $M\in \mathbb{C}^{m\times n}$ is denoted by $M^{\dagger}$. Note that if $MM^*$ is invertible, then $M^{\dagger}$ can be given as $M^{\dagger}= M^*(MM^{*})^{-1}$.

We consider several function spaces and refer to \cite{zhou1996robust}, \cite{lax} for their definitions and the norms on these spaces. The norm on $L_2$ is simply denoted by $\|\cdot\|$. The norm on $H_{\infty}$ is denoted $\|\cdot\|_{\infty}$ and for the sake of clarity, we include its definition
$${\|F\|_{\infty} = \sup_{\textrm{Re}(s)>0} \, \|F(s)\| = \textrm{sup}_{\omega \in \mathbb{R} } \, \|F(j\omega)\|}.$$ The real rational subspace of $H_{\infty}$ is denoted by $RH_{\infty}$ and consists of all proper and real rational stable transfer matrices. 

\section{A recipe for closed-form $H_{\infty}$ optimal control}
\label{sec:recipe}
Consider the following linear and time-invariant system 
\begin{equation} \label{sysss}
\begin{aligned}
\dot{x}(t) &= Ax(t)+Bu(t)+w(t),\\ 
y(t) &= x(t),
\end{aligned}
\end{equation}
for $t\geq 0$, where $x(t)\in \mathbb{R}^n$ is the state of the system, ${u(t) \in \mathbb{R}^m}$ is the control input, $w(t) \in \mathbb{R}^l$ is a disturbance and $y(t) \in \mathbb{R}^{k}$ is the output of the system. In this case the output is the entire state. Moreover, $A$ and $B$ are real-valued matrices of appropriate dimensions. The system \eqref{sysss} can be described in the frequency domain by the following equation 
\begin{equation} \label{gensys}
M(s)\hat{y}=N(s)\hat{u}+\hat{w},
\end{equation}
with $M(s) = sI-A$ and $N(s) = B$ and where $\hat{y}$, $\hat{u}$ and $\hat{w}$ are the Laplace transform of the signals $y$, $u$ and $w$, respectively. We address the following $H_{\infty}$ optimal control problem
\begin{equation} \label{problemhinf}
\inf_{K\in RL_{\infty}} \; \left\|\begin{bmatrix}I\\K \end{bmatrix}(M-NK)^{-1} \right\|_{\infty}.
\end{equation}
It can be written as the following optimization problem
\begin{equation} \label{mainproblem}
\begin{aligned}
\inf_{\substack{ \hat{u} = K\hat{y}, \, K\in RL_{\infty} \text{ stab.}}} \quad &\|\hat{y}\|^2+\|\hat{u}\|^2 \\
\text{subject to} \quad & M(s)\hat{y}=N(s)\hat{u}+\hat{w}\\ \quad& \|\hat{w}\|\leq 1,
\end{aligned}
\end{equation}
only that the optimal value of \eqref{problemhinf} is the squareroot of the optimal value of \eqref{mainproblem}. To clarify, the problem is to design a stabilizing controller $K$, if any exists,  such that the $H_{\infty}$ norm of the closed-loop system from $w$ to $y$ and $u$ is minimized.  

In the following, we will consider any causal and well-posed linear and time-invariant systems from $w$ and $u$ to $y$ that can be written on the form \eqref{gensys} with matrix-valued rational functions $M(s)$ and $N(s)$. However, in many of the examples given in the following section, the system will be on the form \eqref{sysss}. Furthermore, $M(j\omega)$ is assumed to have full column rank for all $\omega \in \mathbb{R}$. If it does not have full column rank, there will exists a problem with lower dimension on $\hat{y}$ with the same performance value. Finally, we assume that the inverse of the complex matrix 
$M(j\omega)M(j\omega)^*+N(j\omega)N(j\omega)^*$
exists for all $\omega \in \mathbb{R}$.
Note that in the case with \eqref{sysss}, the invertibility assumption is fulfilled if the system is stabilizable.\footnote{Stabilizability of \eqref{sysss} is equivalent to that $\begin{bmatrix}sI-A & B \end{bmatrix}$ has full row rank for all $s \in \mathbb{C}_+$, by the PBH test. Thus, $\begin{bmatrix}j\omega I-A & B \end{bmatrix}$ has full row rank for all $\omega \in  \mathbb{R}$ and $(j\omega I-A)(j\omega I-A)^*+BB^*$ is invertible.} The following theorem explicitly suggests a feedback matrix $K$ and the criteria for it to solve \eqref{problemhinf}.

\begin{theo}\label{prop:criteria}
Consider $\omega_0 \in \mathbb{R}$ such that 
$$K \coloneqq -N(j\omega_0)^*M(j\omega_0)\left(M(j\omega_0)^*M(j\omega_0)\right)^{-1}
$$
is real-valued. Then, $K$ solves  \eqref{problemhinf} if $(M-NK)^{-1}\in RH_{\infty}$ and 
$$\omega_0 \in \arg \max_{\omega \in \mathbb{R}} \, \left\|\begin{bmatrix}I\\K \end{bmatrix}(M(j\omega)-N(j\omega)K)^{-1}\right\|.$$
Moreover, if $K$ is a feasible solution, the optimal value is $$\left\|\left(M(j\omega_0)M(j\omega_0)^*+N(j\omega_0)N(j\omega_0)^*\right)^{-1}\right\|^{\frac{1}{2}}.$$
\end{theo}

\begin{remark} \label{rem:controllerform}
Note that if $M$ is square, then $${K = -N(j\omega_0)^*M(j\omega_0)^{-*}}.$$
\end{remark}

The proof of Theorem~\ref{prop:criteria} is given at the end of this section. For now, note that the control law suggested by Theorem~\ref{prop:criteria} is on closed form. Hence, it is in no need of numerical computation to be synthesized, which is generally the case for $H_{\infty}$ optimal control laws. Furthemore, its sparsity pattern is related to the structure of the system. In the case of square and invertible $M(j\omega_0)$, it is directly related to the sparsity of the matrices $M(j\omega_0)^{-1}$ and~$N(j\omega_0)$, see Remark~\ref{rem:controllerform}. 

Theorem~\ref{prop:criteria} is here initially illustrated by means of the following two scalar examples. 
\begin{example} \label{ex:basic1}
Consider \eqref{gensys} with $M(s) = s+1$ and $N(s) = 1$, i.e., it can be written on the form \eqref{sysss} with $A = -1$ and $B = 1$. Given $\omega_0 = 0$, the proposed controller is ${K = -N(0)M(0)^{-1} = -1}$. By Theorem~\ref{prop:criteria}, it solves \eqref{problemhinf} as $$(M-NK)^{-1} = \frac{1}{s+2}$$ is stable and  
\begin{equation*}
\arg \max_{\omega \in \mathbb{R}} \, \left\|\begin{bmatrix}1\\K \end{bmatrix}(M(j\omega)-N(j\omega)K)^{-1}\right\| = 
\arg \max_{\omega \in \mathbb{R}} \; \frac{1}{\omega^2+4}= 0.
\end{equation*}
\end{example}

\begin{example} \label{ex:basic2}
Consider \eqref{gensys} with $M(s) = (s+2)^2$ and ${N(s)=s+1}$. The system can be written as follows in the time domain
$$\ddot{y}+4\dot{y}+4y =\dot{u} + u + w. $$
Given $\omega_0 = 0$, the proposed controller is ${K = -N(0)M(0)^{-1} = -1/4}$. Note that it stabilizes \eqref{gensys} as $$(M-NK)^{-1} = \frac{4}{4s^2+17s+17}.$$ Furthermore, 
\begin{multline*}
\arg \max_{\omega \in \mathbb{R}} \, \left\|\begin{bmatrix}1\\K \end{bmatrix}(M(j\omega)-N(j\omega)K)^{-1}\right\| = \\
\arg \max_{\omega \in \mathbb{R}} \; \frac{1}{(4\omega^2+17)^2+17\omega^2}= 0.
\end{multline*}
Hence, by Theorem~\ref{prop:criteria}, $K= -1/4$ solves \eqref{problemhinf} for this system. 
\end{example}
In the following section we give examples of systems on the form \eqref{gensys} for which the criteria of Theorem~\ref{prop:criteria} are fulfilled. In fact, we show that it is the case for models for temperature dynamics in buildings, water irrigation systems, buffer networks and certain electrical power systems. Naturally, these systems have sparse matrix functions $M$ and $N$ that will be shown to result in a sparse feedback matrix. There, it becomes evident that the suggested control law is suitable for control of large-scale systems.  This section is here ended with the proof of Theorem~\ref{prop:criteria}. 

\begin{proof}[{Proof of Theorem \ref{prop:criteria}}]
The proof is divided into two parts. In the first part, it is shown that no controller can achieve a lower value of \eqref{mainproblem} than 
$$\left\|\left(M(j\omega_0)M(j\omega_0)^*+N(j\omega_0)N(j\omega_0)^*\right)^{-1}\right\|^{\frac{1}{2}}.$$
In the second part, it is shown that if $$K= -N(j\omega_0)^*M(j\omega_0)\left(M(j\omega_0)^*M(j\omega_0)\right)^{-1}$$
is real-valued such that $(M-NK)^{-1}\in RH_{\infty}$ and 
$$ \arg \max_{\omega \in \mathbb{R}} \, \left\|\begin{bmatrix}I\\K \end{bmatrix}(M(j\omega)-N(j\omega)K)^{-1}\right\| = \omega_0,$$
it meets this lower bound. 

Denote the optimal gain by $\gamma_{\text{opt}}$, i.e., 
$$\gamma_{\text{opt}} \coloneqq \inf_{K\in RL_{\infty}} \; \left\|\begin{bmatrix}I\\K \end{bmatrix}(M-NK)^{-1} \right\|_{\infty}.$$
If there does not exists a stabilizing controller $\hat{u} = K\hat{y}$, ${K\in RL_{\infty}}$, such that $(M-NK)^{-1}\in RH_{\infty}$, then $\gamma_{\text{opt}} = \infty$. In this case, the lower bound holds trivially. The remaining part of the proof considers the case when there exists a $K$ such that $(M-NK)^{-1}\in RH_{\infty}$. The optimal gain $\gamma_{\text{opt}}$ can be written as
\begin{equation*} 
\inf_{K\in RL_{\infty}} \; \sup_{\omega \in \mathbb{R}} \;\left\|\begin{bmatrix}I\\K(j\omega) \end{bmatrix}(M(j\omega)-N(j\omega)K(j\omega))^{-1} \right\|
\end{equation*}
by the definition of the $H_{\infty}$ norm. Given $\omega \in \mathbb{R}$, define 
\begin{equation} \label{defz}
\hat{z} = \begin{bmatrix}I\\K(j\omega) \end{bmatrix}(M(j\omega)-N(j\omega)K(j\omega))^{-1} \hat{w}
\end{equation}
for $\hat{w} \in RH_2$, $\|\hat{w}\| \leq 1$. Then, 
\begin{equation} \label{frequencynorm}
\left\|\begin{bmatrix}I\\K(j\omega) \end{bmatrix}(M(j\omega)-N(j\omega)K(j\omega))^{-1} \right\| = \begin{cases}\sup_{\|\hat{w}\| \leq 1} \; \|\hat{z}\|,\\ \text{subject to \eqref{defz}}.\end{cases}
\end{equation}
The optimization problem \eqref{frequencynorm} can equivalently be written as
\begin{align}
\sup_{\|\hat{w}\| \leq 1} &\quad\|\hat{z}\|\\ \text{subject to} & \quad \begin{bmatrix} M(j\omega) & -N(j\omega)\end{bmatrix} \hat{z} = \hat{w} \label{systemconstraint}\\ 
&\quad \hat{z} = \begin{bmatrix}\hat{y}\\\hat{u} \end{bmatrix}\label{defiz}\\
& \quad \hat{u}=K(j\omega)\hat{y}\label{controllerconstraint},
\end{align}
where intermittent signals $\hat{u}$ and $\hat{y}$ have been introduced. 
Now, if we drop the controller constraint \eqref{controllerconstraint}, and implicitly \eqref{defiz}, $\gamma_\text{opt}$ can be lower bounded as follows
$$
\gamma_{\text{opt}} = \inf_{K\in RL_{\infty}} \; \sup_{\omega \in \mathbb{R}} \;\sup_{\|\hat{w}\| \leq 1 \atop \eqref{systemconstraint}, \, \eqref{defiz},\,  \eqref{controllerconstraint}} \|\hat{z}\| \geq \sup_{\omega \in \mathbb{R}} \;\sup_{\|\hat{w}\|_{H_2} \leq 1 \atop \eqref{systemconstraint}} \|\hat{z}\|_{L_2}$$
where the latter problem is equivalent to
$$\left \| \begin{bmatrix} M & -N\end{bmatrix}^{\dagger}\right\|_{L_{\infty}} =\sup_{\omega \in \mathbb{R}} \; \| \begin{bmatrix} M(j\omega) & -N(j\omega)\end{bmatrix}^{\dagger}\|.
$$
Given the invertibility assumption on the matrix $$M(j\omega_0)M(j\omega_0)^*+N(j\omega_0)N(j\omega_0)^*,$$ we can write
\begin{multline*}
\begin{bmatrix} M(j\omega) & -N(j\omega)\end{bmatrix}^{\dagger} \\= \begin{bmatrix} M(j\omega_0)^* \\ -N(j\omega_0)^*\end{bmatrix}\left(M(j\omega_0)M(j\omega_0)^*+N(j\omega_0)N(j\omega_0)^*\right)^{-1},
\end{multline*}
and it follows that 
\begin{equation*}
\left \| \begin{bmatrix} M & -N\end{bmatrix}^{\dagger} \right\|_{L_{\infty}} = \left\|\left(M(j\omega_0)M(j\omega_0)^*+N(j\omega_0)N(j\omega_0)^*\right)^{-1}\right\|^{\frac{1}{2}}.
\end{equation*}

Consider \eqref{gensys} and $\omega_0 \in \mathbb{R}$ such that 
$$K = -N(j\omega_0)^*M(j\omega_0)\left(M(j\omega_0)^*M(j\omega_0)\right)^{-1}
$$
is real-valued and $(M-NK)^{-1} \in RH_{\infty}$. Then, we can compute
\begin{equation*}
\left\|\begin{bmatrix}I\\K \end{bmatrix}(M-NK)^{-1} \right\|_{\infty} = \sup_{\omega \in \mathbb{R}} \, \left\|\begin{bmatrix}I\\K \end{bmatrix}(M(j\omega)-N(j\omega)K)^{-1}\right\|. \end{equation*}
Further, from the assumption
$$\omega_0 \in \arg \max_{\omega \in \mathbb{R}} \, \left\|\begin{bmatrix}I\\K \end{bmatrix}(M(j\omega)-N(j\omega)K)^{-1}\right\|$$
we arrive at
\begin{equation*}
\left\|\begin{bmatrix}I\\K \end{bmatrix}(M-NK)^{-1} \right\|_{\infty}= \left\|\begin{bmatrix}I\\K \end{bmatrix}(M(j\omega_0)-N(j\omega_0)K)^{-1}\right\|.
\end{equation*}
Inserting the expression for $K$ yields 
\begin{multline*}
\left\|\begin{bmatrix}I\\K \end{bmatrix}(M(j\omega_0)-N(j\omega_0)K)^{-1}\right\|=\\ \left\|\begin{bmatrix}M(j\omega_0)^*\\-N(j\omega_0)^* \end{bmatrix}\left(M(j\omega_0)M(j\omega_0)^*+N(j\omega_0)N(j\omega_0)^*\right)^{-1}\right\| \\ 
\left\|\left(M(j\omega_0)M(j\omega_0)^*+N(j\omega_0)N(j\omega_0)^*\right)^{-1}\right\|^{\frac{1}{2}},
\end{multline*}
which is exactly the lower bound. 
\end{proof}

\section{Bottleneck frequencies} \label{sec:bottleneck}

Theorem~\ref{prop:criteria} is suggestive of the following recipe for synthesis of $H_{\infty}$ optimal static controllers; \textit{Given $\omega_0 \in \mathbb{R}$, check that the proposed $K$ 
\begin{enumerate}
\item stabilizes \eqref{gensys} and that $(M-NK)^{-1}$ is a proper and real rational transfer matrix,
\item achieves the maximum of $$\, \left\|\begin{bmatrix}I\\K \end{bmatrix}(M(j\omega)-N(j\omega)K)^{-1}\right\|$$
at $\omega_0$.
\end{enumerate}}

This section treats important cases of when this recipe follows through. That is, systems and frequencies $\omega_0$ for which the criteria of Theorem~\ref{prop:criteria} are satisfied. Furthemore, we comment on how the optimal controller can be used in control of large-scale systems.

\subsection{Important cases when $\omega_0 = 0$}\label{sec:freqzero}
Consider \eqref{gensys} with $M(s) = Es-A$, where $E,\, A \in \mathbb{R}^{k\times k}$, $A^{-1}$ exists, and $N(s)=B$, where $B \in \mathbb{R}^{k\times m}$. Then, the following Corollary follows from Theorem~\ref{prop:criteria}.
\begin{coro} \label{coro:general}
Consider \eqref{gensys} with $M(s) = Es-A$, $A^{-1}$ exists, where $E,\, A \in \mathbb{R}^{k\times k}$, and $N(s)=B$, where $B \in \mathbb{R}^{m\times k}$. Define
$$K =B^TA^{-T}.$$
Then, \eqref{problemhinf} is solved by $K$
if the minimal realization of ${(E,A+BBA^{-T})}$ is Hurwitz and 
 $$\omega^2FG^{-1}F^T+j\omega(F^T-F)+G \succeq \|G^{-1}\|^{-1} I,$$
holds for all $\omega \in \mathbb{R}$, where $F = EA^T$ and ${G = AA^T+BB^T}$. 
\end{coro}
\begin{proof}
Given $M(s) = Es-A$, where $E,\, A \in \mathbb{R}^{k\times k}$, and $N(s)=B$, where $B \in \mathbb{R}^{m\times k}$ we have that $$KM(0)^*=-N(0)^*$$ which yields $K = B^TA^{-T}$. Thus, by Theorem~\ref{prop:criteria}, $K$ solves \eqref{problemhinf} if $(Es-A-BB^TA^{-T})^{-1}$ is stable and  
\begin{equation} \label{seccrit}
0 \in \arg \max_{\omega \in \mathbb{R}} \, \left\|\begin{bmatrix}A^T\\B^T \end{bmatrix}(EA^Tj\omega-AA^T-BB^T)^{-1}\right\|.
\end{equation}
Note that $(Es-A-BB^TA^{-T})^{-1}$ is stable if and only if the minimal realization of  ${(E,A+BB^TA^{-T})}$ is Hurwitz. Further, \eqref{seccrit} is equivalent to that
\begin{equation*}
\left\|\begin{bmatrix}A^T\\B^T \end{bmatrix}(EA^Tj\omega-AA^T-BB^T)^{-1}\right\| \leq \left\|(AA^T+BB^T)^{-1}\right\|^{\frac{1}{2}},
\end{equation*}
holds for all $\omega \in \mathbb{R}$. This criterion can be written as 
\begin{multline*}
\left(\begin{bmatrix}A^T\\B^T \end{bmatrix}(EA^Tj\omega-AA^T-BB^T)^{-1} \right)^* \begin{bmatrix}A^T\\B^T \end{bmatrix}(EA^Tj\omega-AA^T-BB^T)^{-1} \\ \preceq \left\|(AA^T+BB^T)^{-1}\right\|.
\end{multline*}
If we denote $F = EA^T$ and $G = AA^T+BB^T$, it can be equivalently written as 
\begin{equation*}
\left((Fj\omega-G)G^{-1}(-F^Tj\omega-G) \right)^{-1} \preceq \|G^{-1}\|
\end{equation*}
and further to 
\begin{multline*}
\|G^{-1}\|^{-1} \preceq (Fj\omega-G)G^{-1}(-F^Tj\omega-G) = \omega^2FG^{-1}F^T+j\omega(F^T-F)+G
\end{multline*}
which is exactly the inequality given in the statement, and the proof is done.
\end{proof}
The criteria given in Corollary~\ref{coro:general} can be verified by the following numerical tests: 
\begin{enumerate}
\item$(E,A+BB^TA^{-T})$ is Hurwitz if and only if there exists a $P \succ 0$ such that 
\begin{equation*}
(A+BB^TA^{-T})^TPE\\ +E^TP(A+BB^TA^{-T}) \prec 0.
\end{equation*}
\item Compute 
\begin{equation*}\left\|\begin{bmatrix}I\\B^TA^{-T} \end{bmatrix}\left(Es-A-BB^TA^{-T}\right)^{-1}\right\|_{\infty}.
\end{equation*}
If it is equal to $\left\|(AA^T+BB^T)^{-1}\right\|^{\frac{1}{2}}$
then 
 $$\omega^2FG^{-1}F^T+j\omega(F^T-F)+G \succeq \|G^{-1}\|^{-1} I,$$
holds for all $\omega \in \mathbb{R}$. 
\end{enumerate}

The numerical tests described above are as expensive to perform as general $H_{\infty}$-synthesis and are without any guarantee of an applicable controller. However, if the tests follow through, this approach gives an applicable optimal controller on closed-form. This is illustrated in the following example. 

\begin{example}[Asymmetric system] \label{ex:vehicle}
Consider 
\begin{equation} \label{toyformation}
\begin{aligned}
\begin{bmatrix}\dot{x}_1\\\dot{x}_2 \end{bmatrix} &= \underbrace{\begin{bmatrix}-a & a \\ 0 & -1 \end{bmatrix}}_{=A}\begin{bmatrix}x_1 \\ x_2 \end{bmatrix}+\underbrace{\begin{bmatrix}1 \\ 0 \end{bmatrix}}_{=B}u + w.
\end{aligned}
\end{equation}
We will now show that the control law suggested by Corollary~\ref{coro:general} is optimal for this system when $a>0$. First, we need to show that $A+BB^TA^{-T}$ is Hurwitz. It is clearly the case as the eigenvalues of 
$$A+BB^TA^{-T} = \begin{bmatrix}-a-1/a &a\\ 0 &-1 \end{bmatrix}$$
have negative real-part for $a>0$. Secondly, we need to show that the matrix inequality in Corollary~\ref{coro:general} holds for all $\omega \in \mathbb{R}$. In this specific case it is equivalent to that 
\begin{align}
\begin{bmatrix}\omega^2\frac{a^2}{a^2+1}+2a^2+1-g^{-1} & -a(1-j\omega)\\ -a(1+j\omega) & \omega^2+1-g^{-1} \end{bmatrix} & \succeq 0, \label{thineq2}
\end{align}
holds for all $\omega\in\mathbb{R}$ where 
\begin{equation*}
g \coloneqq \left\|\begin{bmatrix}I\\B^TA^{-T} \end{bmatrix}\left(A+BB^TA^{-T}\right)^{-1}\right\| \\= \frac{1}{a^2+1}\left \|\begin{bmatrix}1 & a\\ a & 2a^2+1 \end{bmatrix} \right \|. 
\end{equation*}
Note that, by the definition of $g$, we know that 
\begin{align*}
\begin{bmatrix}2a^2+1-g^{-1} & -a\\ -a &1-g^{-1} \end{bmatrix} & \succeq 0,
\end{align*}
which equivalent to $2a^2+1-g^{-1} \geq0$, $1-g^{-1}\geq0$ and $(1-g^{-1})(2a^2+1-g^{-1})-a^2\geq0$ by the Schur complement lemma. Equation \eqref{thineq2} holds if and only if 
\begin{equation}
\begin{aligned}
\omega^2+1-g^{-1}\geq 0,\\ 
\omega^2\frac{a^2}{a^2+1}+2a^2+1-g^{-1}\geq0,\\
\omega^4\frac{a^2}{a^2+1}+ \omega^2\left(3a^2+(1-g^{-1})\left(1+\frac{a^2}{a^2+1}\right)\right)\\+(2a^2+1-g^{-1})(1-g^{-1})-a^2\geq 0,\end{aligned}
\end{equation}
for all $\omega \in \mathbb{R}$, again by the Schur complement lemma. It is easy to see that the first two inequalities hold by the definition of $g$ and the terms $\omega^2$ and $a^2/(a^2+1)$ being non-negative. The coefficients in the latter inequality are all non-negative, again by the definition of $g$, while $\omega$ appears as $\omega^2$ and $\omega^4$. Thus, it holds for all $\omega$ and the proof is done. The optimal control signal is given by $u = B^TA^{-T}x,$ i.e., $u = -x_1/a$.
\end{example}

The structural properties of the control law are suitable for control of large-scale systems. This is showcased by the following example.

\begin{example}[Water irrigation]  \label{ex:water}
Consider the following model of the pool dynamics in a water irrigation network
\begin{equation} \label{trans}
\begin{aligned} 
\dot{q}_i &= \frac{1}{\alpha_i} \left[r_i -u_{i-1}-\beta_i q_i -w_i \right],\\ 
\dot{r}_i &= \frac{1}{\tau_i}\left[-r_i +u_i\right],
\end{aligned}
\end{equation}
where $i = 1,\dots, N$, and $N$ is the total number of pools. Moreover, $q_i$ is the water level in pool $i$, around an operating point while $u_i$ is the inflow to pool $i$, and $u_{i-1}$ is the outflow of pool $i$.  Note that $u_0 = 0$, i.e., there is no outflow from pool $1$. The signal $w_i$ is an unknown disturbance or uncertainty in the load profile. Parameters $\alpha_i$, $\beta_i$ and $\tau_i$ are all positive. This model is an approximated version\footnote{ Time-delays are present in the model stated in \cite{mareels2005systems}. Furthermore, we will disregard the upper limit constraint on the control input considered in  \cite{mareels2005systems}.}  of that stated in \cite{mareels2005systems}. The pools are in-line, see \figref{water}. 

\begin{figure}
\begin{center}
\includegraphics{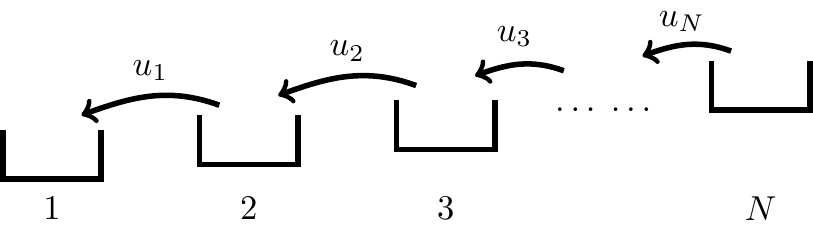}
\caption{Water irrigation network in Example~\ref{ex:water}.}
\label{water}
\end{center}
\end{figure}

The overall system of $N$ pools can be written as 
\begin{align*}
\dot{x} = Ax+Bu+Hw,
\end{align*}
where $x = [q_1,r_1,q_2,r_2,\dots, q_N,r_N]^T$, $u = [u_1, \dots, u_N]^T$ and $w = [w_1,\dots,w_N]^T$. Furthermore, $A$ is block-diagonal with $N$ number of $2\times 2$ elements, with element $i$ given by
\begin{align*}
(A)_{ii} = \begin{bmatrix}
-\beta_i/\alpha_i &1/\alpha_i \\ 0& -1/\tau_i
\end{bmatrix}.
\end{align*}
Moreover, column $i$ of the control input matrix $B$ is given by 
\begin{align*}
(B)_{\cdot i} = \begin{cases} 
\begin{bmatrix}0^{N-1} & \frac{1}{\tau_N} \end{bmatrix}^T & \textrm{if $i = N$}, \\ 
\begin{bmatrix} 0^{2i-1} & \frac{1}{\tau_i} &-\frac{1}{\alpha_{i+1}} &0^{2(N-i)-3}\end{bmatrix} ^T& \textrm{ otherwise},
\end{cases}
\end{align*}
while $H$ is of dimension $2N\times N$ with $(H)_{ii} = 1$ if $i$ is odd, otherwise zero. Moreover, the type of system described can be used as a model for transportation networks, where a commodity is transported from one place in the network to another. 

Consider the problem of designing each control signal $u_{i}$ based only on local measurements, i.e., $q_i$, $r_i$, $q_{i+1}$, such that the impact that the disturbances are optimally attenuated with minimum control effort. Corollary~\ref{coro:general} suggests one such control law as
\begin{align*}
u_{i} = \begin{cases} -q_N/\beta_N-r_N& \text{if } i = N,\\ -q_i/\beta_i-r_i+q_{i+1}/\beta_{i+1}& \text{otherwise},\end{cases}
\end{align*}
that is optimal for a wide range of system parameters $\alpha_i$, $\beta_i$ and~$\tau_i$. In terms of the water irrigation system, this control law will keep the water levels of each pool around their operating point with minimum control effort. 
\end{example}

In the following set of examples, the criteria in Corollary~\ref{coro:general} can easily be shown to be fulfilled. Thus, no tests are needed and the controller can be readily implemented. Consider \eqref{gensys} with $M(s) = Es-A$, where $E,\, A \in \mathbb{R}^{n\times n}$, and $N(s)=B$, where $B \in \mathbb{R}^{m\times n}$. The following corollary declares the properties of the system's matrices under which $K=B^TA^{-T}$ is optimal. The main properties are symmetry of $A$ and $E$.

\begin{coro} \label{coro:sym}
Consider \eqref{gensys} with $M(s) = Es-A$, where $E,\, A \in \mathbb{R}^{n\times n}$ and $N(s)=B$, where $B \in \mathbb{R}^{m\times n}$. Moreover, $E=E^T \succ 0$, $A = A^T \prec 0$ and $EA=AE$. Then, ${K = B^TA^{-1}}$ solves \eqref{problemhinf}.
\end{coro}
\begin{proof}
Consider $P = -A^{-1}E^{-1}= -E^{-1}A^{-1} \succ 0$. Then, 
\begin{multline*}
(A+BB^TA^{-1})^TPE+EP(A+BB^TA^{-1}) = -2I-2A^{-1}BB^TA^{-1} \prec 0
\end{multline*}
and thus $(E,A+BB^TA^{-1})$ is Hurwitz. Further, by Corollary \ref{coro:general}, $K = B^TA^{-1}$ solves \eqref{problemhinf} if  
 $$\omega^2FG^{-1}F^T+j\omega(F^T-F)+G \succeq \|G^{-1}\|^{-1} I,$$
where $G = A^2+BB^T$ and $F = EA$, holds for all $\omega \in \mathbb{R}$. Note that this inequality is equivalent to 
$$\omega^2FG^{-1}F^T+G \succeq \|G^{-1}\|^{-1} I,$$
due to that $F^T = F$. It is easily seen that the latter inequality holds for all $\omega \in \mathbb{R}$. 
\end{proof}

\begin{example}[Temperature dynamics of a building]
When $E \succ 0$ and $A \prec 0$, the system \eqref{gensys} with $M(s) = Es-A$ covers the dynamics of temperature in a building. The average temperature in a room $i$, denoted $T_i$, as governed by Fourier's law of thermal conduction, is given by
\begin{equation} \label{ex:TempControl}
\begin{aligned}
m_ic\dot{T}_i = \,& p_i(T_{out}-T_i)+\sum_{j\in \mathcal{N}_i}p_{ij}\left(T_j-T_i\right)+u_i+d_i,
\end{aligned} 
\end{equation}
where $m_i$ is the air mass of room $i$ and $c$ is the specific heat capacity of air. Furthermore, $\mathcal{N}_i$ is the set of rooms that share a wall/floor/ceiling with room $i$. The heat conduction  coefficients  $p_i$ and $p_{ij}$ are constant, real-valued and strictly positive while $T_{out}$ is the outdoor temperature. Inputs $d_i$ and $u_i$ are disturbance and control inputs, respectively. Disturbances occur, e.g., when a  window is opened. It is assumed that the average temperatures of each room can be measured as well as controlled, the latter through heating and cooling devices modelled by the control inputs~$u_i$.  

The overall system can be written as $E\dot{x}= Ax+u+w$
where $(x)_i = T_i$, $(u)_i = u_i$, $E$ is a diagonal matrix with positive elements $(E)_{ii} = cm_{i}$, $A \prec 0$ and $(w)_i \coloneqq d_i+p_iT_{out}$. Hence, it is on the form \eqref{gensys} with $M(s)=Es-A$, $A\prec 0$ and $E \succ 0$, and $N(s) = I$.

Consider the problem of designing the control input vector $u$ given the entire state $x$ such that the impact, as measured by the $H_{\infty}$ norm, from disturbance $w$ on $x$ and $u$.  The problem described is a $H_{\infty}$ optimal state feedback problem without any structural requirement on the controller to be designed. However, Corollary \ref{coro:sym} gives a solution to the problem as 
\begin{align*} 
 u = B^T\nu \text{ where } A\nu = x,
\end{align*} 
which can easily be implemented even when the system is of large order. 
\end{example}

The following example treats buffer systems where $E = I$ and $A$ is diagonal. Hence, the inverse of $A$, as it appears in $K = B^TA^{-1}$, is easily computed. 
\begin{example}[Distributed control of buffer systems] \label{ex:buffer}
Consider a system comprised of three buffers with diffusive dynamics
\begin{equation} \label{buffer}
\begin{aligned}
\dot{x}_1 &= -x_1 + u_{12}-u_{21}+w_1\\ 
\dot{x}_2 &= -2x_2 +u_{21}-u_{12} +u_{23}-u_{32}+w_2\\ 
\dot{x}_3 &=-3x_3 +u_{32}-u_{23}+w_3
\end{aligned} 
\end{equation}
where each state $x_i \in \mathbb{R}$, the control inputs $u_{ij} \in \mathbb{R}$ and the disturbances $w_i \in \mathbb{R}$. The overall system can be depicted by a line graph with three nodes representing the subsystems and links between the nodes that share control inputs.

Define vectors $(y)_i = x_i$, $(u)_k = u_{ij}$ and $(w)_i = w_i$.  Consider the problem of designing each control signal $u_{ij}$ based only on measurements of the states $x_i$ and $x_j$ such that the impact on the cost $\|y\|^2+\|u\|^2$ from the disturbance $w$, for $\|w\|\leq 1$, is minimized. That is, to construct each control signal of only local measurements such that the disturbances are optimally attenuated with minimum control effort. Corollary \ref{coro:sym} gives one such control law as 
\begin{align*}
&u_{12} = x_1-\frac{x_2}{2}, &u_{21} = -u_{12} \\
&u_{23} = \frac{x_2}{2}-\frac{x_3}{3}, &u_{32} = -u_{23}.
\end{align*}
\end{example}

The result for a general network of $N$ diffusive components, as the example treated in the introduction, is interesting in its own right and therefore included in the following proposition.
\begin{prop} \label{prop:buffer}
Consider a system comprised of $N$ subsystems
\begin{equation} \label{buffer}
\begin{aligned}
\dot{x}_i &= -a_ix_i + \sum_{(i,j) \in {\mathcal{E}}}u_{ij}-u_{ji}+w_i,
\end{aligned} 
\end{equation}
 where $a_i > 0$ for all $i =1,\dots, N$ and $(i,j)$ is in the set $\mathcal{E}$ if and only if subsystems $i$ and $j$ are connected. Define $(w)_i = w_i$, $(y)_i = x_i$ and $(u)_k = u_{ij}$. Then, the control law
\begin{align*}
u_{ij} = -x_i/a_i+x_j/a_j
\end{align*}
solves \eqref{problemhinf}.
\end{prop}
\begin{proof}
The overall system can be written as $\dot{x} = Ax+Bu+w$ where $A$ is diagonal with $A_{ii} = -a_i$ and $B \in \mathbb{R}^{n \times 2| \mathcal{E}|}$ with
$$B_{ij} =1, B_{ji}=-1 \text{ if }(i,j) \in \mathcal{E}$$. Note that $A\prec 0$ as $a_i<0$ for all $i = 1,\dots,N$. Thus it follows from Corollary~\ref{coro:sym} that ${u = B^TA^{-1}x}$ is optimal, i.e., 
$$u_{ij} = -x_i/a_i+x_j/a_j,$$
and the proof is done.
\end{proof}
The control law given by Proposition~\ref{prop:buffer} comply with the restrictions on information sharing. Also, it distributes the impact of a disturbance between the subsystems, through the links of the network, in relation to the rate of diffusion of the subsystems. The proposed control law is distributed as each control input $u_{ij}$ is only comprised of states $i$ and $j$, i.e., local information. Furthermore, it is scalable, as even if the set $\mathcal{E}$ changes, i.e., a link is added or removed, it will not change the already existing control inputs. 

Similarly to the systems treated in Proposition~\ref{prop:buffer}, consider a system of $N$, now multivariable, subsystems 
\begin{equation*}
\begin{aligned}
\dot{x}_i &= A_ix_i + \sum_{(i,j) \in {\mathcal{E}}}B_{ij}u_{ij}+w_i, i\in (1,\dots,N)
\end{aligned} 
\end{equation*}
where each $A_i\in \mathbb{R}^{n_i\times n_i}$ are symmetric and negative definite while $B_{ij}\in \mathbb{R}^{n_i}$ and $n_i \in \mathbb{N}$ is the order of subsystem $i$. This system fulfil the requirements of Corollary~\ref{coro:sym} and the optimal control law is
\begin{align*}
u_{ij} = B_{ij}^T\left(A_i^{-1}x_i+A_j^{-1}x_j \right).
\end{align*}
Note that each control input $u_{ij}$ is only comprised of states $i$ and $j$, i.e., local information. Again, even if the set $\mathcal{E}$ changes, i.e., a link is added or removed, it will not change the already existing control inputs.

In the previous examples we have treated $M(s)$ of the form $M(s) = Es-A$. For general $M(s)$ and $N(s)$, the criteria for optimality are those stated in Theorem~\ref{prop:criteria}, however, for $\omega_0 = 0$. 

\subsection{The case of $\omega_0\neq 0$} \label{sec:nonzerofrequency}
In this section we will cover an application for which the optimal controller is defined at $\omega_0 \neq 0$. 

\begin{example}
Consider the linearized model of the synchronous machine
$$\frac{1}{\omega_0^2}\ddot{\theta}+\frac{2\zeta}{\omega_0}\dot{\theta}+\theta = P_{\text{dist}}+P_{\text{gen}},$$
where $P_{\text{gen}}$ is the mechanical power-input, $P_{\text{dist}}$ is a disturbance in this input, i.e., a load, while $\theta$ is the rotor phase angle of the machine. Furthermore, $\omega_0, \; \zeta >0$. 

We would like to design a stabilizing droop control ${P_{\text{gen}} = K\dot{\theta}}$ to minimize the cost $$\|P_{\text{gen}}\|^2+\|\dot{\theta}\|^2$$
over the set of disturbances $\|P_{\text{dist}}\|\leq 1$. Define
$$M(s) = \frac{1}{\omega_0^2}s+\frac{2\zeta}{\omega_0}+\frac{1}{s}$$
and note that 
$$M(s)^{-1} = \frac{\omega_0^2 s}{s^2+2\zeta\omega_0s+\omega_0^2}$$
is strictly proper. The dynamics can be written as 
$$M(s)\hat{\theta} = \hat{P}_\text{dist}+\hat{P}_{\text{gen}}.$$
The droop control gain $$K = -M(j\omega_0)^{-1}=-\frac{\omega_0}{2\zeta}$$
is optimal as
\begin{multline*}
(M(s)-K)^{-1} = \frac{\omega_0^2s}{s^2+s(2\zeta\omega_0-K\omega_0^2)+\omega_0^2}
= \frac{\omega_0^2s}{s^2+s(4\zeta^2+\omega_0^2)\omega_0/2\zeta+\omega_0^2}
\end{multline*}
is stable and 
\begin{multline*}
\arg \max_{\omega \in \mathbb{R}} \left |\frac{j\omega_0^2\omega}{-\omega^2+j\omega(4\zeta^2+\omega_0^2)\omega_0/2\zeta+\omega_0^2}\right |^2 \\ 
\arg \max_{\omega \in \mathbb{R}} \frac{\omega_0^2\omega^2}{(\omega_0^2-\omega^2)^2+\omega^2(4\zeta^2+\omega_0^2)^2\omega_0^2/4\zeta^2}
= \omega_0.
\end{multline*}
Now, consider a network of $N$ synchronous machines described by 
$$m\ddot{\theta}+d\dot{\theta}+L\theta = P_{\text{dist}}+P_{\text{gen}},$$
where  $m$ and $d$ are constants while $L$ is the weighted Laplacian of the network's graph, see \cite{newman2010networks}. Denote the eigenvalues of $L$ by $\lambda_i$, $i = 1\dots N$. It is well-known that $\lambda_i \geq 0$. Thus, $L$ is positive definite and can be given the eigendecomposition $P^{-1}D P$, where $P$ is a unitary matrix and $D$ is a diagonal matrix with $D_{ii} = \lambda_i$. Denote $x = P\theta$, $u = PP_{\text gen}$ and $w = PP_{\text dist}$. Then, the system can be written as a set of $N$ scalar subsystems
$$m\ddot{x}_i+d\dot{x}_i+\lambda_ix_i = w_i+u_i.$$
For each subsystem treated separately, $u_i = -\dot{x}_i/d$ is optimal. It follows that, $P_{\text gen} = P^{-1}u = -d^{-1}P^{-1}\dot{x} = -d^{-1}\dot{\theta}$ is optimal for the original system. More generally, we could consider a systems of the form 
$$M\ddot{\theta}+D\dot{\theta}+L\theta = P_{\text{dist}}+P_{\text{gen}},$$
where $M$ and $D$ are now matrices. If there exists a uniform matrix $Q$ that simultaneously diagonalizes $M$, $D$ and $L$, such a system can be treated in a similar manner. 
\end{example}

\section{Scope of the results}\label{sec:scope}

This section comments on the problem setups covered and relates the considered performance objective to classical synthesis objectives in robust control. Moreover, it is shown how the controller suggested in Theorem~\ref{prop:criteria} is derived.  

\subsection{Problem setups covered}
It is possible to consider a more general cost functions than that in \eqref{mainproblem}. However, in favour of clarity, the simplest case was treated in Theorem~\ref{prop:criteria}. The general form is described next. The statement given can be proven in the same manner as Theorem~\ref{prop:criteria}.

Consider \eqref{mainproblem}, however with the objective function $${\|Q\hat{y}\|^2+\|\hat{u}\|^2},$$ where $Q$ is a real-valued matrix with full row-rank and of appropriate dimension. The row rank assumption is natural as otherwise the cost would have abundant terms. It follows from the assumption that $QQ^T$ is invertible. 

Define $Q^{\dagger} = Q^T(QQ^T)^{-1}$. Then, 
\begin{equation*} 
\inf_{K\in RL_{\infty}} \; \left\|\begin{bmatrix}Q\\K \end{bmatrix}(M-NK)^{-1} \right\|_{\infty}  
\end{equation*}
is solved by 
$$K \coloneqq -N(j\omega_0)^*M(j\omega_0)\left(M(j\omega_0)^*M(j\omega_0)\right)^{-1}Q^TQ
$$
if $(M-NK)^{-1}\in RH_{\infty}$ and 
$$\omega_0 \in \arg \max_{\omega \in \mathbb{R}} \, \left\|\begin{bmatrix}Q\\K \end{bmatrix}(M(j\omega)-N(j\omega)K)^{-1}\right\|.$$
Moreover, if $K$ is a feasible solution, the optimal value is $$\left\|\left(M(j\omega_0)Q^\dagger Q^{\dagger *}M(j\omega_0)^*+N(j\omega_0)N(j\omega_0)^*\right)^{-1}\right\|^{\frac{1}{2}}.$$

The set of problems that can be treated through the problem formulation
\begin{equation*} 
\inf_{K\in RL_{\infty}} \; \left\|\begin{bmatrix}Q\\K \end{bmatrix}(M-NK)^{-1} \right\|_{\infty}
\end{equation*}
will now be specified. To begin, denote 
\begin{equation}
z = \begin{bmatrix}Qy\\u \end{bmatrix}\label{z}.
\end{equation} 
Further, consider full row-rank matrices $B_1$, $C_1$ and $C_2$ and define $C_2^{\dagger} = C_2^T(C_2C_2^T)^{-1}$. Then, \eqref{gensys} with $M(s) = (sI-A)^{-1}C_2^{\dagger}$, $N(s) = B_2$, $\hat{w}=B_1\hat{d}$ and \eqref{z} with $Q = C_1C_2^{\dagger}$ can be written in state space as
\begin{align*}
\dot{x} &= Ax+B_1d+B_2u,\\ 
z &= \begin{bmatrix} C_1 \\ 0
\end{bmatrix} x+\begin{bmatrix} 0 \\ I
\end{bmatrix}u,\\ 
y &= C_2x.
\end{align*}
Note that given $B_1 \in \mathbb{R}^{k\times q}$, $B_1\hat{d}$ with $\hat{d}\in H^q_2$ spans the entire $H^k_2$. The objective function is now $\|\hat{z}\|^2$. If $C_2=I$, this is the general full-information $H_{\infty}$ control problem, see \cite[Ch. 4]{stoorvogel1992h}. Its dual problem, that of filtering, can be treated through minimization of the $H_{\infty}$ norm of the transposed closed-loop transfer matrix, i.e., 
$$\left( M^T-K^TN^T\right)^{-1}\begin{bmatrix}I& K^T \end{bmatrix},$$
in the case with $Q = I$. See \cite{lidstromDiscrete} for a treatment of this. In the case with $C_2 \neq I$, note that the formulation is not that of the simplified $H_{\infty}$ output feedback problem as there is no measurement noise, see e.g. \cite[Ch. 14.2]{zhou1996robust}. 

This subsection is ended with a comment on the functions $M(s)$ and $N(s)$ in \eqref{gensys} and solvability of the considered problem. Assume that $M(s)$ is square and invertible. Then, we can write \eqref{gensys} as 
$$\hat{y} = M(s)^{-1}N(s)\hat{u}+M(s)^{-1}\hat{w}.$$
Note that if $M(s)^{-1}$ is not stable, we can never stabilize the system with $\hat{u}$, as the signal $\hat{w}$ is unknown. Furthermore, for the purpose of realizability of the system, $M(s)^{-1}$ and $M(s)^{-1}N(s)$ needs be proper.

\subsection{The performance objective} 
Consider the performance objective \eqref{problemhinf}. It will now be related to classical performance requirements in robust control. To make the comparison clear, consider \eqref{gensys} with $M$ such that $M^{-1}$ exists. Denote ${M^{-1} = P}$ and set $N(s)= -I$. The system can then be written~as 
$\hat{y} = P(s)(-\hat{u}+\hat{w}).$
Moreover, \eqref{problemhinf} becomes
\begin{equation} \label{problemhinf2}
\inf_{K\in RL_{\infty}} \; \left\|\begin{bmatrix}P(I+KP)^{-1} \\KP(I+KP)^{-1} \end{bmatrix} \right\|_{\infty}.
\end{equation}

In general, closed-loop performance is concerned with the behaviour of the systems corresponding to the four transfer functions $(I+PK)^{-1}$, $(I+PK)^{-1}P$, $K(I+PK)^{-1}$ and $(I+PK)^{-1}PK$. See \cite{zhou1996robust} for more details on this statement. The considered performance objective \eqref{problemhinf2} implies properties on two of the transfer functions. However, in the problems we consider, $P$ often has the characteristics of a low-pass filter. Thus, small $$\|P(I+KP)^{-1}\|_{\infty}$$ implies that $\|(I+P(j\omega)K(j\omega))^{-1}\|$ is small at low frequencies. This is generally the performance requirement aimed for.  Also, in this example, as $P$ and $K$ are square with the same dimensions, we have that 
$$\|K(I+P(j\omega)K(j\omega))^{-1}\| \leq  \|K(j\omega)\|\|(I+P(j\omega)K(j\omega))^{-1}\|.$$
Thus, $\|(I+P(j\omega)K(j\omega))^{-1}\|$ small for low frequencies implies that $${\|K(j\omega)(I+P(j\omega)K(j\omega))^{-1}}\|$$ is small for low frequencies, as long as $\|K(j\omega)\|$ is kept small. Note that, as we have shown static controllers to be optimal, the requirement on $K$ implies $\|K\|$ to be small.

\subsection{Construction of optimal controller \\and how to choose $\omega_0$}
In the proof of Theorem~\ref{prop:criteria}, the following inequality is shown to hold
\begin{multline*} 
\inf_{K\in RL_{\infty}} \; \left\|\begin{bmatrix}I\\K \end{bmatrix}(M-NK)^{-1} \right\|_{\infty}  \geq  \sup_{\omega \in \mathbb{R}} \;\sup_{\|\hat{w}\|_{H_2} \leq 1 \atop  [M(j\omega) \, -N(j\omega)] \hat{z} = \hat{w}} \|\hat{z}\|_{L_2}.
\end{multline*}
The problem on the right-hand-side of the inequality can be solved by a least-squares argument. Given a $\hat{w} \in H_2$ and a $\omega \in \mathbb{R}$ consider 
\begin{align*}
\sup &\quad\|\hat{z}\|_{L_2}\\ \text{subject to} & \quad \begin{bmatrix} M(j\omega) & -N(j\omega)\end{bmatrix} \hat{z} = \hat{w}. 
\end{align*}
This is a least-squares problem with unique solution
\begin{multline*}
\hat{z}_{\text{opt}}= \begin{bmatrix} M(j\omega) & -N(j\omega)\end{bmatrix}^{\dagger}\hat{w} \\ = \begin{bmatrix}M(j\omega)^* \\ -N(j\omega)^* \end{bmatrix}\left(M(j\omega)M(j\omega)^*+N(j\omega)N(j\omega)^*\right)^{-1}\hat{w},
\end{multline*}
as $M(j\omega)M(j\omega)^*+N(j\omega)N(j\omega)^*$ is assumed to be invertible for all $\omega \in \mathbb{R}$. 
It follows that 
$$\sup_{\|\hat{w}\|_{H_2}\leq 1 \atop  [M(j\omega) \, -N(j\omega)] \hat{z} = \hat{w}} \|\hat{z}\|_{L_2} = \sup_{\|\hat{w}\|_{H_2}\leq 1} \|\hat{z}_{\text{opt}}\|_{L_2}.$$
Furthermore, consider
$$\hat{w}_0 \in \arg \max_{\|\hat{w}\|_{H_2} \leq 1} \|\hat{z}_{\text{opt}}\|_{L_2},$$
and 
$$\omega_0 \in \arg \max_{\omega\in \mathbb{R}} \left\|[M(j\omega) \, -N(j\omega)]^{\dagger}\hat{w}_0\right\|.$$
Then, the lower bound can be written as 
$$\| \begin{bmatrix} M & -N\end{bmatrix}^{\dagger}\|_{L_{\infty}}  = \| \hat{z}_0\|_{L_2}.$$ 
where $\hat{z}_0$ is defined as follows
\begin{equation} \label{zzero}
\hat{z}_{0} = \begin{bmatrix}M(j\omega_0)^* \\ -N(j\omega_0)^* \end{bmatrix}\left(M(j\omega_0)M(j\omega_0)^*+N(j\omega_0)N(j\omega_0)^*\right)^{-1}\hat{w}_0.
\end{equation}

The lower bound, i.e., $\hat{z}_0$, is suggestive of a control law of the form $\hat{u} = K\hat{y}$. To see this, first define
\begin{equation} 
\hat{x}_0 = \left(M(j\omega_0)M(j\omega_0)^*+N(j\omega_0)N(j\omega_0)^*\right)^{-1}\hat{w}_0.
\end{equation}
Then, $\hat{z}_0$ in \eqref{zzero} can be written as
$$\hat{z}_{0} = \begin{bmatrix}M(j\omega_0)^* \\ -N(j\omega_0)^* \end{bmatrix}\hat{x}_0.$$
Consider the definition of $\hat{z}$ in the variables $\hat{y}$ and $\hat{u}$ as well as the controller description made in \eqref{defiz} and \eqref{controllerconstraint}, respectively, in the proof of Theorem~\ref{prop:criteria}. These equations together with $\hat{z}_0$, as described above, suggests that
$$\begin{bmatrix} \hat{y}\\ \hat{u} \end{bmatrix} = \begin{bmatrix}M(j\omega_0)^* \\ -N(j\omega_0)^* \end{bmatrix}\hat{x}_0,$$
i.e., $KM(j\omega_0)^* = -N(j\omega_0)^*$. Note that, for a given $\omega_0 \in \mathbb{R}$, $K$ could be complex valued. In practice, such a controller is not implementable. Therefore, only real-valued $K$ are considered, as enforced by the constraint $K\in RL_{\infty}$. However, given $K \in \mathbb{C}^{m\times k}$ that fulfils $KM^*(j\omega_0) = -N^*(j\omega_0)$ one could construct a dynamic controller $K(s)$ through interpolation for which $K(j\omega_0)M^*(j\omega_0) = -N^*(j\omega_0)$. However, we do not consider that in the results presented here. 

Note that the closed-loop system with the proposed controller achieves its $L_2$-gain at frequency $\omega_0$, if the controller is optimal. From a disturbance point of view this means that the worst-case disturbances are those of frequency $\omega_0$. In order to be able to follow the recipe described in Section~\ref{sec:recipe}, one needs to choose a $\omega_0 \in \mathbb{R}$, from which $K$ is defined. If $\omega_0$ is chosen from the set 
\begin{equation}
\arg \max_{\omega \in \mathbb{R}} \| \begin{bmatrix} M(j\omega) & -N(j\omega)\end{bmatrix}^{\dagger}\|,
\end{equation}
 the criterion
$$\omega_0 \in \arg \max_{\omega \in \mathbb{R}} \, \left\|\begin{bmatrix}I\\K \end{bmatrix}(M(j\omega)-N(j\omega)K)^{-1}\right\|,$$
in Theorem~\ref{prop:criteria} is implicitly fulfilled. Then, one only needs to check that $(M-NK)^{-1} \in RH_{\infty}$. In fact, the latter criterion is then necessary and sufficient for $K$ to be optimal.

Lower bounds, similar to that considered in the proof of Theorem~\ref{prop:criteria}, i.e., 
\begin{equation*} 
\inf_{K\in RL_{\infty}} \; \left\|\begin{bmatrix}I\\K \end{bmatrix}(M-NK)^{-1} \right\|_{\infty} \geq \left\|\begin{bmatrix} M & -N\end{bmatrix}^{\dagger}\right\|_{L_{\infty}}, 
\end{equation*}
can be constructed for systems and problems not necessarily on the form consider in this paper. However, we are interested in the class of systems for which the inequality is in fact an equality. For the considered systems and problem form, this class includes many interesting engineering problems, as was made evident through examples in the previous section. Further, a lower bound as the one given above, can be used to investigate fundamental limitations of the performance in networked systems. This was utilized in \cite{pates2017control} to state fundamental limitations of the performance of vehicle~platoons.

\section{Comparison to Other Methods for Distributed Control and General $H_{\infty}$ synthesis}
\label{sec:discussion}

This section compares the presented results to general $H_{\infty}$ synthesis and some existing approaches to distributed control.

\subsection{Comparison to general purpose numerical synthesis}
In \cite{doyle1989state}, a state-space approach for general $H_{\infty}$ synthesis was first stated. With this method, one can compute the solution to the $H_{\infty}$ synthesis problem numerically. This is in general the case for conventional methods to $H_{\infty}$ synthesis, see e.g., \cite{stoorvogel1992h,dullerud2013course,zhou1996robust,doyle1989state,gahinet1994linear}. 

It is not easy to relate the numerical solution given by the approach in \cite{doyle1989state} to the structure of the considered system. In fact, almost regardless of the sparsity of the system, the synthesized controller is often dense. This is in contrast to the solution we propose. We will illustrate this by means of a simple example with a small order system. However, the point we try to make is true also for systems of larger order. 

Consider \eqref{gensys} with 
$M(s) = sI-A$ and $N(s)=B$ where 
\begin{align*}
A = \begin{bmatrix} -1 & 0 & 0\\ 0 & -3 & 0\\ 0 & 0 & -2\end{bmatrix} \text{ and } B = \begin{bmatrix}
-1 & 0 &0\\
1 &1 &-1\\
0 &0 &1
\end{bmatrix}.
\end{align*}
This system fulfils the requirements of Corollary~\ref{coro:sym} and the control law ${u = K x}$ with 
\begin{align*}
K = B^TA^{-1}= \begin{bmatrix}
1 &-\frac{1}{3} &0\\
0 &-\frac{1}{3} &0\\
0 &\frac{1}{3} &-\frac{1}{2}
\end{bmatrix}
\end{align*}
minimizes the closed-loop norm. The following state feedback matrix, denoted $K_{\text{num}}$, is also a feasible solution to the considered problem,
\begin{displaymath}
  K_{\text{num}}=\begin{bmatrix}
   0.93 & -0.11&
  0.00 \\   -0.05 &
   -0.17 &   -0.01 \\
  0.04 &
  0.16 &
  -0.26
\end{bmatrix}.
\end{displaymath}
However, it is numerically derived by the approach in \cite{doyle1989state}, i.e., bisection over an ARE-constraint until the minimal value of the norm is obtained. Note that $K_{\textrm{num}}$ is less sparse than $K=B^TA^{-1}$. 

\subsection{Comparison to $H_{\infty}$ control with structural constraints}

One approach to the problem of distributed control is to include structural constraints as part of the synthesis problem. In other words, to only consider a subset of the full class of stabilizing controllers $K$, that have certain structural properties. However, given a performance measure, this approach does not have to result in a convex optimization problem and, of course, not a globally optimal control law. 

Consider a system where $K = -N(j\omega_0)^*M(j\omega_0)^{-*}$ has been shown to be optimal. For clarity, we consider systems where $M$ is invertible. Then the following equality holds 
\begin{equation*}
\inf_{K\in RL_{\infty}} \; \left\|\begin{bmatrix}I\\K \end{bmatrix}(M-NK)^{-1} \right\|_{\infty} = \inf_{K\in \mathcal{S}} \; \left\|\begin{bmatrix}I\\K \end{bmatrix}(M-NK)^{-1} \right\|_{\infty}.
\end{equation*}
where \begin{multline*}\mathcal{S} = \{ K \in RL_{\infty}\; : \; K \textrm{ has the same zero pattern as }-N(j\omega_0)^*M(j\omega_0)^{-*} \}.\end{multline*}
That is, we could add this structural constraint to the optimization problem without loss of global optimality. Thus, our results illustrates one type of sparsity pattern for which global optimality is preserved, for the considered class of systems. 

\subsection{Comparison to synthesis for positive systems}
There are connections between the class of systems we consider and that of externally positive systems. Every externally positive system is also positively dominated. If $G$ is the transfer matrix of a positively dominated system then $\|G\|_{\infty} = \|G(0)\|$. This property is utilized in \cite{rantzer2015scalable} to solve the problem of $H_{\infty}$ control. It is shown that positively dominated system can be optimally controlled with static controllers that result from solving a linear program. The class of problems we consider in Section~\ref{sec:freqzero} shares this property, but gives naturally sparse controllers and involves no optimization. Internally positive systems are a subset of the externally positive systems. They are great candidates for models in e.g., biological and transportation systems. That is, any application where the quantity modelled is naturally nonnegative. It is well known that a system represented by the transfer function $G(s) = C(sI-A)^{-1}B+D$ is internally positive if and only if $A$ is a Metzler-matrix and $B,C,D\geq 0$. The closed-loop system of the buffer network described in Example~\ref{ex:buffer},  or more generally the system treated in Proposition~\label{prop:buffer}, is internally positive. That is, the optimal controller renders an internally positive closed-loop system. 

\subsection{Comparison to spatially invariant systems}
In \cite{bamieh2002distributed}, it is shown that solutions to $H_{\infty}$-type AREs for spatially invariant systems are translation invariant. Thus, the resulting control law is also spatially invariant. This is true for the control law we propose as well, as will be illustrated by the following simple example. Consider
\begin{equation} \label{sysSI}
\dot{x} = \underbrace{\begin{bmatrix} -3 & 1 & 0 & 1 \\ 1 & -3 & 1 & 0 \\ 0 &1 & -3 & 1 \\ 1 & 0 &1 & -3 \end{bmatrix}}_{\eqqcolon A} x+u+w,
\end{equation}
with $x,u,w \in \mathbb{R}^4$. It is spatially invariant by the definition in \cite{bamieh2002distributed} as $A$ is translation invariant. Notice that $A=A^T \prec 0$. The optimal control law proposed by Corollary~\ref{coro:sym} is given by 
$u = A^{-1}x$, which is also spatially invariant. 

While the theory in \cite{bamieh2002distributed} specifies when the system's property of spatial invariance is preserved in the control law, it does not give closed-form expressions of the control law. This is in contrast to our work. However, the spatially invariant control laws in \cite{bamieh2002distributed} can be effectively computed by a family of low order problems across the spatial frequency.  Further, our results are not restricted to spatially invariant systems. On the other hand, \cite{bamieh2002distributed} covers additional performance measures to the $H_{\infty}$ norm. The results presented in this paper as well as those given in \cite{bamieh2002distributed} showcase controllers that provide \textit{global} performance guarantees and stability. This is in contrast to the approaches to distributed control previously covered. 

\section{Conclusions and future works}
We give a simple closed-form expression for a $H_{\infty}$ feedback law and the criteria for it to be optimal. It is illustrated through examples, that the proposed control law is suitable for control of large-scale systems. Systems for which the criteria are readily fulfilled are specified. 

The results do not treat the problem of $H_{\infty}$ output feedback. However, it is natural to ask if an optimal control law can be given on a simple closed form also in this case. Attempts suggests that the structure of the controller is quite involved. Further investigations are needed to determine if there exist close to optimal controllers with similar sparsity patterns as the ones treated here. Further directions for extensions of the presented work includes to consider non-linearities, such as saturation constraints and time delays, as well as time-invariant~systems.

\section{Acknowledgement of Financial Support}
This research was supported by the Swedish Research Council through the LCCC Linnaeus Center and by the Swedish Foundation for Strategic Research through the project ICT-Psi. 

\bibliographystyle{plain}
\bibliography{ref2.bib,ref.bib}

\end{document}